\documentclass[12pt]{amsart}
\usepackage{amsmath, amsthm, amscd, amsfonts, amssymb, graphicx, color}
\usepackage[bookmarksnumbered, colorlinks, plainpages]{hyperref}

\newtheorem{theorem}{Theorem}[section]

\newtheorem{proposition}[theorem]{Proposition}

\theoremstyle{definition}
\newtheorem{definition}[theorem]{Definition}
\newtheorem{example}[theorem]{Example}

\theoremstyle{remark}
\newtheorem{remark}[theorem]{Remark}
\numberwithin{equation}{section}

\begin{document}
\setcounter{page}{1}

\title[ $J$-frame sequences in Krein Space]{ $J$-frame sequences in Krein Space}

\author[S. Karmakar, Sk M. Hossein and K. Paul]{Shibashis Karmakar$^1$, Sk. Monowar Hossein$^2$ and Kallol Paul$^3$}

\address{$^{1}$ Department of Mathematics, Jadavpur University, Jadavpur-32, West Bengal, India.}
\email{\textcolor[rgb]{0.00,0.00,0.84}{shibashiskarmakar@gmail.com}}

\address{$^{2}$ Department of Mathematics, Aliah University, IIA/27 New Town, Kolkata-156, West Bengal, India.}
\email{\textcolor[rgb]{0.00,0.00,0.84}{sami$\_$milu@yahoo.co.uk}}

\address{$^{3}$ Department of Mathematics, Jadavpur University, Jadavpur-32, West Bengal, India.}
\email{\textcolor[rgb]{0.00,0.00,0.84}{kalloldada@gmail.com}}



%\date{Received: xxxxxx; Revised: yyyyyy; Accepted: zzzzzz.
%\newline \indent $^{*}$ Corresponding author}

\begin{abstract}
Let $\{f_n:n\in\mathbb{N}\}$ be a $J$-frame for a Krein space ${\textbf{\textit{K}}}$ and  $P_M$ be a $J$-orthogonal projection from ${\textbf{\textit{K}}}$ onto a subspace $M$. In this article we find sufficient conditions under which $\{P_M(f_n):n\in\mathbb{N}\}$ is  a $J$-frame for $P_M\textbf{\textit{K}}$ and $\{(I-P_M)f_n\}_{n\in{\mathbb{N}}}$ is a $J$-frame for $(I-P_M)\textbf{\textit{K}}$. We also introduce $J$-frame sequence for a Krein space ${\textbf{\textit{K}}}$ and study some properties of $J$-frame sequence analogues to Hilbert space frame theory.
\end{abstract} 
\maketitle

\section{Introduction and preliminaries}

Let $\textbf{\emph{H}}$ be a real or complex Hilbert space. A sequence $\{f_n:n\in\mathbb{N}\}$ is said to be a frame for a Hilbert space $\textbf{\emph{H}}$ if there exists positive reals $A$ and $B$ s.t. $A\|f\|^2\leq\sum_{n\in\mathbb{N}}|\langle{f},f_n\rangle|^2\leq{B\|f\|^2}$ for all $f\in{\textbf{\emph{H}}}$. Frame for Hilbert spaces was defined by Duffin and Schaeffer \cite{ds} in 1952 to study some problems in nonharmonic Fourier series. Daubechies \textit{et. al.} \cite{dgm} published a landmark paper in 1986 in this direction while working on wavelets and signal processing. After their work, the theory of frames begun to be more widely studied. Powerful tools from Operator theory and Banach space theory are being used to study frames in Hilbert spaces and it produced some deep results in Frame theory.
Many researchers studied frame in different aspects and applications \cite{oc, pgc, pgcgt, pcs, phl}. Krein space theory \cite{jb} is also rich in application among many areas of Mathematics. Some known areas of application are in High energy physics, Quantam cosmology, Krein space filtering etc. In the year 2011, Giribet\textit{ et. al.} \cite{gmmm} introduced frame
in Krein spaces which is known as $J$-frame. Recently Esmeral \textit{et. al.} \cite{koe} defined frames in Krein spaces in a more general setting.

The projection methods play a central role in almost every branches of Mathematics including Hilbert space frame theory. This method helps in evaluating truncation error which arises in computing approximate solutions to moment problems, as well as handling the very difficult problem of computing dual frames \cite{ochr}. The following result plays an  important role in Hilbert space frame theory:
\begin{theorem}\cite{oc}
Let $\{f_n:n\in\mathbb{N}\}$ be a frame for a Hilbert space $\textbf{\emph{H}}$ with lower frame bound $A$ and upper frame bound $B$. Also let $P$ be an orthogonal projection from $\textbf{\emph{H}}$ onto $V$. Then $\{P(f_n):n\in\mathbb{N}\}$ and $\{(I-P)(f_n):n\in\mathbb{N}\}$ are also frames for $V$ and $V^\perp$ respectively, both admitting the same optimal frame bounds.

Conversely, if $\{f_n\}_{n\in{\mathbb{N}}}$ is a frame for $P(\textbf{\emph{H}})$ and $\{g_n\}_{n\in{\mathbb{N}}}$ is a frame for $(I-P)\textbf{\emph{H}}$, both with frame bounds $0<A\leq{B}$, then $\{f_n\}_{n\in{\mathbb{N}}}\bigcup{\{g_n\}_{n\in{\mathbb{N}}}}$ is a frame for $\textbf{\emph{H}}$ admitting the same frame bounds.
\end{theorem}

There are two definitions of frame in Krein space, one introduced by  Giribet \textit{et al.} \cite{gmmm} and the other by Esmeral \textit{et al.} \cite{koe}. The above Theorem 1.1 does not hold in Krein space if we follow 
 the definition of Giribet \textit{et. al.} \cite{gmmm}, we have given a counterexample in (3.3) to show that. But the Theorem 1.1 holds for Krein space if we follow the definition of Esmeral \textit{et al.} \cite{koe}.

In this article we provide an example in (2.10) to show that definition in \cite{koe} involves fundamental symmetry of a Krein space, which is not unique. We show that   the definition in  \cite{gmmm} is independent of fundamental symmetry and so  we use the definition given in \cite{gmmm} for our further study.

Let $P_M$ be a $J$-orthogonal projection from a Krein space $\textbf{\textit{K}}$ onto a subspace $M$. We find sufficient condition under which $\{P_M(f_n):n\in\mathbb{N}\}$ is a $J$-frame for the subspace $(M,[~,~])$.  We also introduce $J$-frame sequence for a Krein space $\textbf{\textit{K}}$ and derive some of its important properties similar to what had been done in Hilbert space frame theory.

\section{Basic Definitions}
We briefly mention the definitions, geometric interpretations and some basic properties of Krein spaces that we need for our study \cite{jb,isty,ando}.
\begin{definition}
An abstract vector space $(\textbf{\textit{K}},[~,~])$ that satisfies the following requirements is called a Krein space.
\begin{enumerate}
  \item $\textbf{\textit{K}}$ is a linear space over the field $F$, where $F$ is either $\mathbb{R}$ or $\mathbb{C}$.
  \item there exists a bilinear form $[~,~]\in{F}$ on \textbf{\textit{K}} such that
	\begin{equation*}
	\begin{split}
  [y,x] &=\overline{[x,y]}\\
  [ax+by,z] &=a[x,z]+b[y,z]
	\end{split}
	\end{equation*}
  for any $x,y,z\in{\textbf{\textit{K}}}$, $a,b\in{F}$, where $\overline{[~,~]}$ denote the complex conjugation.
  \item The vector space $\textbf{\textit{K}}$ admits a canonical decomposition $\textbf{\textit{K}}=\textbf{\textit{K}}^+[\dot{+}]\textbf{\textit{K}}^-$ such that $(\textbf{\textit{K}}^+,[~,~])$ and $(\textbf{\textit{K}}^-,-[~,~])$ are Hilbert spaces relative to the norms $\|x\|=[x,x]^{\frac{1}{2}}(x\in{\textbf{\textit{K}}^+})$ and $\|x\|=(-[x,x]^{\frac{1}{2}})(x\in{\textbf{\textit{K}}^-})$.
\end{enumerate}
\end{definition}
Now every canonical decomposition of $\textbf{\textit{K}}$ generates two mutually complementary projectors $P^+$ and $P^-$ ($P^{+}+P^-=I$, the identity operator on $\textbf{\textit{K}}$ ) mapping $\textbf{\textit{K}}$ onto $\textbf{\textit{K}}^+$ and $\textbf{\textit{K}}^-$ respectively. Thus for any $x\in{\textbf{\textit{K}}}$, we have $P^{\pm}=x^{\pm}$, where $x^+\in{\textbf{\textit{K}}^+}$ and $x^-\in{\textbf{\textit{K}}^-}$. The projectors $P^+$ and $P^-$ are called canonical projectors. They are also ortho-projectors i.e. they are orthogonal (self-adjoint) projection operator.

The linear operator $J:\textbf{\textit{K}}\to{\textbf{\textit{K}}}$ defined by the formula $J=P^+-P^-$ is called the canonical symmetry of the Krein space $\textbf{\textit{K}}$. The operator $J$ is a self-adjoint, unitary and involutory operator. The canonical symmetry $J$ immediately generates orthogonal canonical orthoprojectors $P^{\pm}$ according to the formulae $P^{\pm}=\frac{1}{2}(I\pm{J})$ and a canonical decomposition $\textbf{\textit{K}}=\textbf{\textit{K}}^+\oplus\textbf{\textit{K}}^-,~\textbf{\textit{K}}^{\pm}=P^{\pm}\textbf{\textit{K}}$ and also the $J$-metric defined by the formula $[x,y]_{J}=[x,Jy]$, where $x,y\in{\textbf{\textit{K}}}$. The vector space $\textbf{\textit{K}}$ associated with the $J$-metric is a Hilbert space, called the associated Hilbert space of the Krein space $\textbf{\textit{K}}$.
\begin{definition}
A subspace $M$ of a Krein space $\textbf{\textit{K}}$ is said to be $J$-positive if $[x,x]\geq{0}~(x\in{M})$. Further $M$ is said to be uniformly $J$-positive if there is a $\epsilon{>}0~$such that $[x,x]\geq{\epsilon\|x\|^2}~(x\in{M})$. $J$-negativity and $J$-uniform negativity of a subspace is defined as positivity and $J$-uniform positivity with respect to the selfadjoint involution $-J$. By this definition the trivial subspace $\{0\}$ is uniformly $J$-postive as well as uniformly $J$-negative.
\end{definition}
\begin{definition}
Let $(\textbf{\textit{K}},[~.~],J)$ be a Krein space and let $x,y\in\textbf{\textit{K}}$. We say that $x$ is orthogonal to $y$, denoted by $x\perp{y}$, when $[x,y]_J=0$. Similarly, we say that $x$ is $J$-orthogonal to $y$, denoted by $x[\perp]y$, when $[x,y]=0$.

Since the associated space $(\textbf{\textit{K}},[~,~]_J)$ for a Krein space $\textbf{\textit{K}}$ is a Hilbert space, therefore we can study linear operators acting on Krein spaces. Some topological concepts such as continuity and closure of a set, are concerning to the topology induced by the $J$-norm.
\end{definition}
\begin{definition}
Let $M$ be a closed subspace of a Krein space $\textbf{\textit{K}}$. The subspace $M^{[\perp]}=\{x\in\textbf{\textit{K}}:[x,y]=0,~\textmd{for all}~y\in{M}\}$ is the $J$-orthogonal complement of $M$ with respect to $[~.~]$.
\end{definition}
\begin{definition}
Let $M$ be a subspace of a Krein space $\textbf{\textit{K}}$. $M$ is said to be projectively complete if $\textbf{\textit{K}}=M+M^{[\perp]}$.
\end{definition}

The $J$-adjoint of an operator $T$ in Krein spaces, denoted by $T^{[\ast]}$, satisfies $[T(x),y]=[x,T^{[\ast]}(y)]$. However, such $T$ have an adjoint operator in the associated Hilbert space $(\textbf{\textit{K}},[~,~]_J)$, denoted by $T^{\ast{J}}$. Furthermore, there is a relation between $T^{\ast{J}}$ and $T^{[\ast]}$, which is $T^{[\ast]}=JT^{\ast{J}}J$.

Let $P$ be a bounded linear operator on a Hilbert space $\textbf{\emph{H}}$. Also let $\textit{R}(P)=\textmd{Range of the mapping}~P$ and $\textit{Ker}(P)=\textmd{Kernel of the mapping}~P$. Then $P$ is called a projection if $P^2=P$. If $\textit{R}(P)=M$ and $\textit{Ker}(P)=N$, then $P$ is called the projection on $M$ parallel to $N$. Therefore $(I-P)$ becomes the projection on $N$ parallel to $M$.\\
Here 
\begin{definition}
A linear operator on a Krein space $\textbf{\textit{K}}$ is said to be $J$-selfadjoint if $T=T^{[\ast]}$.
\end{definition}

A $J$-self-adjoint projection is called a $J$-projection. If $P$ is a $J$-projection with $\textit{R}(P)=M$ , then it is necessarily the projection on $M$ parallel to $M^{[\perp]}$. Conversely, given a closed subspace $M$, a projection $P$ on $M$ parallel to $M^{[\perp]}$, if exists, is $J$-self-adjoint, hence a $J$-projection.
\begin{definition}
A subspace $M$ is said to be regular if it is the range of a $J$-projection.
\end{definition}
Every regular subspace is closed. If $M$ is regular with $J$-projection $P$, then its $J$-orthocomplement $M^{[\perp]}$ with $J$-projection $(I-P)$ is also regular.

Let $(\textbf{\textit{K}},[~.~],J)$ be a Krein space. Suppose $\textbf{\textit{F}}=\{f_n:n\in{\mathbb{N}}\}$ is a Bessel sequence of $\textbf{\textit{K}}$ and $T\in{L(\ell^2(I),\textbf{\textit{K}})}~(\textmd{where~}\ell^2(I):=\{(c_i):\sum_{i\in{I}}|c_i|^2<\infty\})$ is the synthesis operator for the Bessel sequence $\textbf{\textit{F}}$. Let $I_+=\{i\in{I}:[f_i,f_i]\geq{0}\}$ and $I_-=\{i\in{I}:[f_i,f_i]<0\}$, then $\ell^2(I)=\ell^2(I_+)\bigoplus{\ell^2(I_-)}$. Also let $P_{\pm}$ denote the orthogonal projection of $\ell^2(I)$ onto $\ell^2(I_{\pm})$. Let $T_{\pm}=TP_{\pm}$, $M_{\pm}=\overline{span\{ f_i:i\in{I_{\pm}}\}}$ then we have $R(T)=R(T_+)+R(T_-)$, where $R(T)$ represents range of the operator $T$.
\begin{definition}\cite{gmmm}
 A Bessel sequence $\textbf{\textit{F}}$ is said to be a $J$-frame for $\textbf{\textit{K}}$ if $R(T_+)$ is a maximal uniformly $J$-positive subspace of $\textbf{\textit{K}}$ and $R(T_-)$ is a maximal uniformly $J$-negative subspace of $\textbf{\textit{K}}$.
\end{definition}
We denote by $Q_M$ and $P_M$ the orthogonal and $J$-orthogonal projections on $M$ respectively \textit{i.e.} $Q_M=Q_M^{\ast{J}}=Q_M^2$ and $P_M=P_M^{[\ast]}=P_M^2$. If a subspace $M~(\subset{\textbf{\textit{K}}})$ is projectively complete then it must be closed \textit{i.e.} $M=\overline{M}$ and non-degenerate \textit{i.e.} $M\cap{M^{[\perp]}}=\{0\}$. From Theorem 7.16 of \cite{isty} we know that if $M$ is projectively complete, then $M$ is regular.

Now we state a result which is due to P. Acosta-Hum$\acute{a}$nez \textit{et. al.}\cite{pahko}.
\begin{proposition}
Let $M$ be a closed subspace of $\textbf{\textit{K}}$. The following statements hold.\\
$(i)~$ If $Q_M$ is an orthogonal projection on $M$, then $P_M=Q_{JM}Q_M$ is a $J$-orthogonal projection on $M$.\\
$(ii)~$ If $P_M$ is an $J$-orthogonal projection on $M$, then $Q_M=P_{JM}P_M$ is a orthogonal projection on $M$.
\end{proposition}
Now we provide an example to show that the definition of frames in Krein spaces as defined by Esmeral \textit{et. al.} \cite{koe} is dependent on fundamental symmetry.
\begin{definition}\cite{koe}
Let $\textbf{\textit{K}}$ be a Krein space. A countable sequence $\{f_n\}_{n\in{\mathbb{N}}}$ is called a frame for $\textbf{\textit{K}}$, if there exist constants $0<A\leq{B}<\infty$ such that
\begin{equation*}
A\|f\|_J^2\leq\sum_{n\in{\mathbb{N}}}|[f_n,f]|^2~{\leq}~B\|f\|_J^2
\end{equation*}
\end{definition}
\begin{example}
 Consider the vector space $\ell^2(\mathbb{N})$ over $\mathbb{R}$. Let $x=(\alpha_1,\alpha_2,\ldots),~y=
(\beta_1,\beta_2,\ldots)~\in\ell^2$; we define $\langle{x},y\rangle=\alpha_1\beta_1-\alpha_2\beta_2+\alpha_3\beta_3-\alpha_4\beta_4\ldots$\\
Then $(\ell^2,\langle\cdot\rangle)$ is a Krein space with fundamental symmetry $J_1:\ell^2\rightarrow\ell^2$ defined by $J_1(\alpha_1,\alpha_2,\ldots)=(\alpha_1,-\alpha_2,\ldots)$, but we can define another fundamental symmetry $J_2$ on $(\ell^2,\langle\cdot\rangle)$ s.t. $J_2:\ell^2\rightarrow\ell^2$ defined by $J_2(\alpha_1,\alpha_2,\alpha_3,\alpha_4,\ldots)=
(2\alpha_1-\sqrt{3}\alpha_2,\sqrt{3}\alpha_1-2\alpha_2,\alpha_3,-\alpha_4,\ldots)$.
Now according to the definition given by K. Esmeral et. al. \cite{koe}, the sequence $\{e_n:n\in{\mathbb{N}}\}$ is a frame for the triple $(\ell^2,\langle\cdot\rangle,J_1)$ with frame bound $A=B=1$, but if we consider the triple $(\ell^2,\langle\cdot\rangle,J_2)$, then the sequence $\{e_n:n\in{\mathbb{N}}\}$ is not a frame with frame bound $A=B=1$. So the definition given in \cite{koe} takes away some important geometries of Parseval frame.

But the definition in \cite{gmmm} is independent on $J$-symmetry of the Krein space $\textbf{\textit{K}}$. Let $(\textbf{\textit{K}},[~.~],J_1)$ and $(\textbf{\textit{K}},[~.~],J_2)$ be two cannonical decompositions of $\textbf{\textit{K}}$. But from Theorem 7.19 \cite{isty} we know that the $J$-norms $\|.\|_{J_1}$ and $\|.\|_{J_2}$, are equivalent. Since $R(T_+)$ and $R(T_-)$ are uniformly $J$-definite, hence the norms generated by the inner products $[~.~]$ and $[~.~]_{J_1}$ are equivalent. Also the norms generated by the inner products $[~.~]$ and $[~.~]_{J_2}$ are equivalent. So $R(T_+)$ and $R(T_-)$ are uniformly $J$-definite in the both $J$-fundamental symmetries. Hence if $\{f_n:n\in{\mathbb{N}}\}$ is a $J$-frame for $(\textbf{\textit{K}},[~.~],J_1)$, then $\{f_n:n\in{\mathbb{N}}\}$ is also $J$-frame for $(\textbf{\textit{K}},[~.~],J_2)$. This establishes our claim. 
\end{example}
\section{Main results}
\subsection{General Properties}
Projections and Orthogonal projections play a special role in many aspects of frame theory \cite{pgcml}. We first prove the following theorem in connection with frames and orthogonal projections \cite{oc} in the context of Krein spaces.
\begin{theorem}
Let $(\textbf{\textit{K}},[~.~],J)$ be a Krein space with fundamental symmetry $J$, and let $P_M$ be a $J$-orthogonal projection from $\textbf{\textit{K}}$ onto $M$. Also let $M$ be a uniformly positive definite subspace of $\textbf{\textit{K}}$. Then if $\{f_n\}_{n\in{\mathbb{N}}}$ is a $J$-frame for $\textbf{\textit{K}}$ then $\{P_M(f_n)\}_{n\in{\mathbb{N}}}$ is also a $J$-frame for $P_M\textbf{\textit{K}}$ and $\{(I-P_M)f_n\}_{n\in{\mathbb{N}}}$ is a $J$-frame for $(I-P_M)\textbf{\textit{K}}$.
\end{theorem}
\begin{proof}
Since $P_M$ is a $J$-orthogonal projection with $R(P_M)=M$, therefore it is necessarily a $J$-orthogonal projection onto $M$ parallel to $M^{[\perp]}$. Hence $(I-P_M)$ is a $J$-orthogonal projection onto $M^{[\perp]}$ parallel to $M$. As $M$ is a image of a $J$-orthogonal projection hence $M$ is a regular subspace of $\textbf{\textit{K}}$ i.e. equivalently $M$ is a projectively complete subspace of $\textbf{\textit{K}}$. So $\textbf{\textit{K}}=M[\dot{+}]M^{[\perp]}$. So we have a fundamental decomposition of $\textbf{\textit{K}}$. Since $M$ is uniformly positive therefore by Theorem 7.1 \cite{jb} we claim that $M$ is a maximal uniformly definite regular subspace of $\textbf{\textit{K}}$.  Also we have $P_M(x)=Q_M(x)=x\textmd{ for all }x\in{M}$. Now $\{f_n\}_{n\in{\mathbb{N}}}$ is a frame for the associated Hilbert space $(\textbf{\textit{K}},[~.~]_J)$, so according to Hilbert space frame theory $\{Q_M(f_n)\}_{n\in{\mathbb{N}}}$ is a frame for $(M,[~.~]_J)$. Now for every $f\in{M}$,
\begin{equation*}
\begin{split}
		[f,Q_M(f_i)]_J &=[f,JP_MJP_M(f_i)]_J \\
		& =[f,P_MJP_M(f_i)]\\
		& =[P_M(f),JP_M(f_i)]\\
		& =[P_Mf,P_Mf_i]_J.
\end{split}
\end{equation*}
According to our assumption $M$ is uniformly definite, hence the norms generated by the inner products $[~.~]$ and $[~.~]_J$ are equivalent. Also $\{P_M(f_n)\}_{n\in{\mathbb{N}}}$ is a Bessel sequence in $\textbf{\textit{K}}$. Hence using Proposition 3.3 \cite{gmmm} we have $0<A\leq{B}$ such that
\begin{equation*}
A[f,f]{\leq}\sum_{n\in{\mathbb{N}}}|[f,P_Mf_i]|^2{\leq}B[f,f]~\textmd{for every}~f\in{M}
\end{equation*}
Hence $\{f_n\}_{n\in{\mathbb{N}}}$ is a frame for $(M,[~.~])$. We condiser the subspace $\{0\}$ as maximal uniformly $J$-negative subspace in $M$, then $\{f_n\}_{n\in{\mathbb{N}}}$ is a $J$-frame for $M$, consisting only positive elements.\\
Since $M$ is maximal uniformly $J$-positive subspace in $\textbf{\textit{K}}$, so $M^{[\perp]}$ is a maximal uniformly $J$-negative subspace.\\
Using the self-adjoint involution $-J$ and the similar arguments as above we can show that $\{(I-P_M)f_n\}_{n\in{\mathbb{N}}}$ is a $J$-frame for $(I-P_M)\textbf{\textit{K}}$.
\end{proof}
We next state the theorem, the proof of which follows in the same way as above.
\begin{theorem}
Let $(\textbf{\textit{K}},[~.~],J)$ be a Krein space with fundamental symmetry $J$, and let $P_M$ be an $J$-orthogonal projection from $\textbf{\textit{K}}$ onto $M$. Also let $M$ is a uniformly negative definite subspace of $\textbf{\textit{K}}$. Then if $\{f_n\}_{n\in{\mathbb{N}}}$ is a $J$-frame for $\textbf{\textit{K}}$ then $\{P_M(f_n)\}_{n\in{\mathbb{N}}}$ is also a $J$-frame for $P_M\textbf{\textit{K}}$ and $\{(I-P_M)f_n\}_{n\in{\mathbb{N}}}$ is a $J$-frame for $(I-P_M)\textbf{\textit{K}}$.
\end{theorem}
Since $M(=P\textbf{\textit{K}})$ is a vector space with an indefinite inner product $[~.~]$, $M$ equipped with $[~.~]$ cannot be neutral or semi-definite since $M^{0}=M{\cap}M^{[\perp]}=0$. We have $\textbf{\textit{K}}=M[\dot{+}]M^{[\perp]}$. So according to Theorem 7.16 \cite{isty}, $(M,[~.~])$ is a Krein space. Let us assume that $M$ is a $J$-definite subspace of $\textbf{\textit{K}}$. But any $J$-definite subspace of $\textbf{\textit{K}}$ is regular iff it is uniformly $J$-definite [\cite{jb}, Theorem 8.2]. Therefore $M$ is a uniformly $J$-definite subspace of $\textbf{\textit{K}}$. Without any loss of generality let us assume that $M$ is a uniformly $J$-positive definite subspace of $\textbf{\textit{K}}$. Then $M^{[\perp]}$ is a uniformly $J$-negative definite subspace of $\textbf{\textit{K}}$. Since $\textbf{\textit{K}}$ is ortho-complemented, therefore according to Lemma 11.4 \cite{jb}, $M$ is maximal $J$-positive definite subspace of $\textbf{\textit{K}}$. So $M$ is maximal uniformly $J$-positive definite subspace of $\textbf{\textit{K}}$ and similarly $M^{[\perp]}$ is maximal uniformly $J$-negative definite subspace of $\textbf{\textit{K}}$. Hence the range of a $J$-orthogonal projection is either maximal uniformly $J$-definite subspace of $\textbf{\textit{K}}$ or $M$ is a Krein space with respect to the induced inner product of $\textbf{\textit{K}}$.

But in general if the subspace contain any non-zero neutral elements then $\{P_M(f_n)\}_{n\in{\mathbb{N}}}$ may not be a $J$-frame for $P_M\textbf{\textit{K}}$. This can be checked by the following example.
\begin{example}
Consider the Vector space $\mathbb{R}^3$ over $\mathbb{R}$. Let $\{e_1,e_2,e_3\}$ be the standard orthonormal basis. Define $[e_i,e_i]=1~~~\textmd{for}~{i=1,3}$ and $[e_2,e_2]=-1$. Also $[e_i,e_j]=0$ for $i\neq{j}$. Then $\mathbb{R}^3$ is a Krein space with respect to the above inner product. Let $P_M$ is a $J$-orthogonal projection from $\mathbb{R}^3$ onto $M$ defined by $P_M(x,y,z)=(x,y,0)~\textmd{for all}~(x,y,z)\in\mathbb{R}^3$. Now consider the following collection of vectors $\{(1,1,-1001),(10,-\frac{1}{200},-5),(0,1,0)\}$. It is a $J$-frame for the given Krein space. But if we apply the $J$-orthogonal projection $P_M$, then the resulting collection is not a $J$-frame in $(M,[~.~])$. Since it contains non-zero neutral elements.
\end{example}

\begin{remark}
 Let $(\textbf{\textit{K}},[~.~],J)$ be a Krein space with fundamental symmetry $J$, and let $P_M$ be a $J$-orthogonal projection from $\textbf{\textit{K}}$ onto $M$. Also let $M$ be a uniformly positive definite subspace of $\textbf{\textit{K}}$. Then if $\{f_n\}_{n\in{\mathbb{N}}}$ is a $J$-frame for $\textbf{\textit{K}}$ with frame bounds $B_2\leq{A_2}<0<A_1\leq{B_1}$ then $\{P_M(f_n)\}_{n\in{\mathbb{N}}}$ is also a $J$-frame for $P_M\textbf{\textit{K}}$ and $\{(I-P_M)f_n\}_{n\in{\mathbb{N}}}$ is a $J$-frame for $(I-P_M)\textbf{\textit{K}}$. But generally we can not write their frame bound in terms of original frame bounds.
\end{remark}
\begin{remark} 
 Let $P_M$ is a $J$-orthogonal projection from a Krein space $(\textbf{\textit{K}},[~.~],J)$ onto a closed subspace $M$. Also let $P_MJ=JP_M$. Then 
\begin{equation*}
\begin{split}
		[P_M(f_n),f_n] &=[P_M(f_n),P_M(f_n)]\\
		& =[P_M(f_n),JP_M(f_n)]_J\\
		& =[P_MJP_M(f_n),f_n]_J\\
		& =[Q_M(f_n),f_n]_J\\
		& =[Q_M(f_n),Q_M(f_n)]_J\\
		& =\|Q_M(f_n)\|^2_J.
\end{split}
\end{equation*}
So we have $[P_M(f_n),f_n]>0$ for all $n\in{\mathbb{N}}$. Also $P_M(f_n)=Q_M(f_n)\neq{0}$ for all $n\in{\mathbb{N}}$. 
\end{remark}
We next prove the following theorem
\begin{theorem}
If $\{f_n\}_{n\in{\mathbb{N}}}$ is a $J$-frame for $P\textbf{\textit{K}}$ and $\{g_n\}_{n\in{\mathbb{N}}}$ is a $J$-frame for $(I-P)\textbf{\textit{K}}$, both with $J$-frame bounds $B_2\leq{A_2}<0<A_1\leq{B_1}$, then $\{f_n\}_{n\in{\mathbb{N}}}\bigcup{\{g_n\}_{n\in{\mathbb{N}}}}$ is a $J$-frame for $\textbf{\textit{K}}$ admitting the same $J$-frame bounds.
\end{theorem}
\begin{proof}
Let $P(\textbf{\textit{K}})=M$, then $M$ is a regular subspace of $\textbf{\textit{K}}$. Let $I_+^1=\{i\in{\mathbb{N}}:[f_i,f_i]>0\}$ and $I_-^1=\{i\in{\mathbb{N}}:[f_i,f_i]<0\}$. Also let $M_+^1=\overline{span}\{f_i:i\in{I_{+}^{1}}\}$ and $M_-^1=\overline{span}\{f_i:i\in{I_-^1}\}$. Since $\{f_n\}_{n\in{\mathbb{N}}}$ is a $J$-frame for $M$, so the subspaces $M_+^1$ and $M_-^1$ are maximal uniformly $J$-definite and $M=M_+^1\bigoplus{M_-^1}$.

 Similarly let $I_+^2=\{i\in{\mathbb{N}}:[g_i,g_i]>0\}$ and $I_-^2=\{i\in{\mathbb{N}}:[g_i,g_i]<0\}$. Then we have We assume $M^{[\perp]}=M_+^2\bigoplus{M_-^2}$, where $M_+^2=\overline{span}\{g_i:i\in{I_{+}^{2}}\}$ and $M_-^2=\overline{span}\{g_i:i\in{I_-^2}\}$. According to our assumption $M$ is a regular subspace of $\textbf{\textit{K}}$, so $\textbf{\textit{K}}=M\bigoplus{M^{[\perp]}}$. Now $\{f_n\}_{n\in{I_+^1}}\bigcup\{g_n\}_{n\in{I_+^2}}$ is the set of all positive vectors in $\textbf{\textit{K}}$. We want to show that their closed linear sum is a maximal uniformly $J$-positive subspace of $\textbf{\textit{K}}$. Now $M_+^1{+}M_+^2$ is a uniformly $J$-positive subspace of $\textbf{\textit{K}}$ since $M_+^1{[\perp]}M_+^2$. It is also a maximal uniformly $J$-positive subspace. Similarly we can show that $\overline{span}\Big\{\{f_n\}_{n\in{I_-^1}}\bigcup\{g_n\}_{n\in{I_-^2}}\Big\}=M_-^1{+}M_-^2$ and $M_-^1{+}M_-^2$ is a maximal uniformly $J$-negative subspace. So we have $\textbf{\textit{K}}=(M_+^1{+}M_+^2){+}(M_-^1{+}M_-^2)$. Now we have to show that $\{f_n\}_{n\in{\mathbb{N}}}\bigcup{\{g_n\}_{n\in{\mathbb{N}}}}$ is a $J$-frame for $\textbf{\textit{K}}$ admitting the same $J$-frame bounds, but this is easy to prove. Let $f\in{M_+^1{+}M_+^2}$, then $f=f_1+f_2$, where $f_1\in{M_+^1}$ and $f_2\in{M_+^2}$. Now $[f,f]=[f_1,f_1]+[f_2,f_2]$ and $|[f,f_i]|^2=|[f_1,f_i]|^2$ for all $i{\in}I_+^1$. Also we have $|[f,g_i]|^2=|[f_2,g_i]|^2$ for all $i{\in}I_+^2$. So for all $f\in{M_+^1{+}M_+^2}$, we have 
\begin{equation*}
A_1[f,f]\leq\sum_{i{\in}I_+^1}|[f,f_i]|^2{+}\sum_{i{\in}I_+^2}|[f,g_i]|^2\leq{B_1[f,f]}
\end{equation*}
Similarly we can show that for all $g\in{M_-^1{+}M_-^2}$
\begin{equation*}
A_2[f,f]\leq\sum_{i{\in}I_-^1}|[f,f_i]|^2{+}\sum_{i{\in}I_-^2}|[f,g_i]|^2\leq{B_2[f,f]}
\end{equation*}
This is our proof.
\end{proof}
Combining Theorems 3.1, 3.2 and 3.6 we get the following theorem
\begin{theorem}
Let $(\textbf{\textit{K}},[~.~],J)$ be a Krein space with fundamental symmetry $J$, and let $P$ be an $J$-orthogonal projection on $\textbf{\textit{K}}$.\\
If $\{f_n\}_{n\in{\mathbb{N}}}$ is a $J$-frame for $\textbf{\textit{K}}$, then $\{Pf_n\}_{n\in{\mathbb{N}}}$ is also a $J$-frame for $P\textbf{\textit{K}}$ and $\{(I-P)f_n\}_{n\in{\mathbb{N}}}$ is a $J$-frame for $(I-P)\textbf{\textit{K}}$ provided $P\textbf{\textit{K}}$ is $J$-definite.

Conversely, if $\{f_n\}_{n\in{\mathbb{N}}}$ is a $J$-frame for $P\textbf{\textit{K}}$ and $\{g_n\}_{n\in{\mathbb{N}}}$ is a $J$-frame for $(I-P)\textbf{\textit{K}}$, both with $J$-frame bounds $B_2\leq{A_2}<0<A_1\leq{B_1}$, then $\{f_n\}_{n\in{\mathbb{N}}}\bigcup{\{g_n\}_{n\in{\mathbb{N}}}}$ is a $J$-frame for $\textbf{\textit{K}}$ admitting the same $J$-frame bounds.
\end{theorem}
\subsection{Frame sequences and their properties}
Frame sequence in a Hilbert space $\textbf{\emph{H}}$ is a well known concept. Kaushik \textit{et. al.} \cite{sgv} published a paper in the year 2008 to show some important properties of frame sequences. We will show that the results they found also holds for frames in Krein Spaces.\\
Now let $\textbf{\textit{K}}$ be a Krein space and $\{f_n:n\in{\mathbb{N}}\}$ be a Bessel sequence of no non-zeo neutral elements in $\textbf{\textit{K}}$. Also let $I_+=\{i\in{I}:[f_n,f_n]>{0}\}$ and $I_-=\{i\in{I}:[f_n,f_n]<0\}$.
\begin{definition}
A sequence $\{f_n:n\in{\mathbb{N}}\}$ in $\textbf{\textit{K}}$ is said to be a frame sequence in $\textbf{\textit{K}}$ if $\overline{span}\{f_n:n\in{I_+}\}$ and $\overline{span}\{f_n:n\in{I_-}\}$ are uniformly $J$-positive and uniformly $J$-negative subspace of $\textbf{\textit{K}}$respectively.
\end{definition}
\begin{definition}
A $J$-frame $\{f_n\}\textmd{ in }{\textbf{\textit{K}}}$ is called exact if removal of an arbitrary $f_n$ renders the collection $\{f_n\}$ no longer a $J$-frame for  $\textbf{\textit{K}}$.
\end{definition}
\begin{definition}
A frame $\{f_n\}\textmd{ in }{\textbf{\textit{K}}}$ is called near exact if it can be made exact by removing finitely many elements from it.\\
Also a near exact frame is called proper if it is not exact.
\end{definition}
We next prove the following theorems for frame sequences in Krein space in analogy with frame sequences in Hilbert space theory.
\begin{theorem}
Let $\{f_n\}$ be any frame of $\textbf{\textit{K}}$ and let $\{m_k\}$ and $\{n_k\}$ be two infinite increasing sequence of $\mathbb{N}$
 with $\{m_k\}\cup\{n_k\}=\mathbb{N}$. Also let $\{f_{m_{k}}\}$ be frame for $\textbf{\textit{K}}$.
Then $\{f_{n_{k}}\}$ is a frame for $\textbf{\textit{K}}$ if and only if \\
$\overline{span}\{f_{n_{k}}:[f_{n_{k}},f_{n_{k}}]>0\}=\overline{span}\{f_{m_{k}}:[f_{m_{k}},f_{m_{k}}]>0\}$\\
 and $\overline{span}\{f_{n_{k}}:[f_{n_{k}},f_{n_{k}}]<0\}=\overline{span}\{f_{m_{k}}:[f_{m_{k}},f_{m_{k}}]<0\}$.
\end{theorem}
\begin{proof}
Since $\{f_n:n\in\mathbb{N}\}$ is a $J$-frame for $\textbf{\textit{K}}$, so let $I_+=\{n\in\mathbb{N}:[f_n,f_n]>{0}\}$ and
$I_-=\{n\in\mathbb{N}:[f_n,f_n]<0\}$. Let $M_+=\overline{span}\{f_n:n\in{I_+}\}$ and $M_-=\overline{span}\{f_n:n\in{I_-}\}$. Here $M_+$ and $M_-$ are maximal uniformly $J$-positive and $J$-negative subspace respectively such that $\textbf{\textit{K}}=M_+\bigoplus{M_-}$. Now, $\{f_{m_{k}}\}$ is also a $J$-frame for $\textbf{\textit{K}}$. So let $I_{+}^{\prime}=\{m_k:[f_{m_{k}},f_{m_{k}}]>{0}\}$ and
$I_-^\prime=\{m_k:[f_{m_{k}},f_{m_{k}}]<0\}$. We also have $M_+^1=\overline{span}\{f_{m_{k}}:m_k\in{I_+}\}=M_+$ and similarly $M_-^1=\overline{span}\{f_{m_{k}}:m_k\in{I_-}\}=M_-$. Now $\{f_{n_{k}}\}$ is a subsequence of $\{f_n:n\in\mathbb{N}\}$. So let $I_{+}^{\prime\prime}=\{n_k:[f_{n_{k}},f_{n_{k}}]>{0}\}$ and $I_-^{\prime\prime}=\{n_k:[f_{n_{k}},f_{n_{k}}]<0\}$. So $\{f_{n_{k}}:n_k\in{I_+^{\prime\prime}}\}\subset{M_+}$ and $\{f_{n_{k}}:n_k\in{I_-^{\prime\prime}}\}\subset{M_-}$. If it is a $J$-frame for $\textbf{\textit{K}}$, then we obviously have $\overline{span}\{f_{n_{k}}:n_k\in{I_+^{\prime\prime}}\}={M_+}$ and $\overline{span}\{f_{n_{k}}:n_k\in{I_-^{\prime\prime}}\}={M_-}$.\\
Conversely if $\overline{span}\{f_{n_{k}}:n_k\in{I_+^{\prime\prime}}\}={M_+}$ and $\overline{span}\{f_{n_{k}}:n_k\in{I_-^{\prime\prime}}\}={M_-}$. Then $\{f_{n_{k}}\}$ is a $J$-frame for $\textbf{\textit{K}}$.
\end{proof}
\begin{theorem}
Let $\{f_n\}$ be any $J$-frame for $\textbf{\textit{K}}$ and let $\{m_k\}$, $\{n_k\}$ be two infinite increasing sequences with $\{m_k\}\cup\{n_k\}=\mathbb{N}$. Let $\textbf{\textit{K}}^{\prime}=\overline{span}\{f_{m_k}\}{\cap}\overline{span}\{f_{n_k}\}$.\\
If $\textbf{\textit{K}}^{\prime}$ is a finite dimensional Krein space, then $\{f_{m_{k}}\}$ and $\{f_{n_{k}}\}$ are $J$-frame sequences for $\textbf{\textit{K}}$.
\end{theorem}
\begin{proof}
$\{f_n:n\in\mathbb{N}\}$ is a $J$-frame for $\textbf{\textit{K}}$, so let $I_+=\{n\in\mathbb{N}:[f_n,f_n]>{0}\}$ and
$I_-=\{n\in\mathbb{N}:[f_n,f_n]<0\}$. Also let $M_+=\overline{span}\{f_n:n\in{I_+}\}$ and $M_-=\overline{span}\{f_n:n\in{I_-}\}$. Then we have $\textbf{\textit{K}}=M_+\bigoplus{M_-}$. Now $\{f_{n_{k}}\}$ is a subsequence of $\{f_n:n\in\mathbb{N}\}$. So let $I_{+}^{\prime}=\{n_k:[f_{n_{k}},f_{n_{k}}]>{0}\}$ and $I_-^{\prime}=\{n_k:[f_{n_{k}},f_{n_{k}}]<0\}$. So $\{f_{n_{k}}:n_k\in{I_+^{\prime}}\}\subset{M_+}$ and $\{f_{n_{k}}:n_k\in{I_-^{\prime}}\}\subset{M_-}$. Now since $\textbf{\textit{K}}^{\prime}$ is a finite-dimensional Krein space so let $\textbf{\textit{K}}^{\prime}=M_+^\prime\bigoplus{M_-^\prime}$, where $M_+^\prime$ and $M_-^\prime$ are maximal uniformly $J$-positive and maximal uniformly $J$-negative subspace respectively. Also $M_+^\prime\subset{M_+}$ and $M_-^\prime\subset{M_-}$.
Now let $\{l_k\}\subset{I_+^\prime}$ be any finite subsequence of $\{n_k\}\subset{I_+^\prime}$ such that
 $M_+^\prime=span\{f_{l_{k}}\}$. Let $\{f_{l_{k}}\}$ be a frame for the Hilbert space $(M_+^\prime,[~.~])$ and $0<A^\prime\leq{B}^\prime$ be the bounds of the frame $\{f_{l_{k}}\}$. Now let $f\in{span\{f_{l_{k}}\}}$ be any element. Now, if $f{[\perp]}M_+^\prime$, then for all $f\in{M_+}$,
\begin{equation*}
\sum|[f,f_n]|^2=\sum|[f,f_{n_k}]|^2\geq~A_1[f,f]~ (\textmd{here}~A_1>0)
\end{equation*}
Also, if $f{\in}M_+^\prime$ then
\begin{equation*}
\sum|[f,f_{n_k}]|^2\geq~\sum|[f,f_{l_k}]|^2\geq~A_1^\prime[f,f]
\end{equation*}
Otherwise, we have
$f=\sum\alpha_kf_{n_k}=f^\prime+f^{\prime\prime}$,\\
here $i{\in}~\{n_k\}\setminus\{l_k\},~j{\in}~\{l_k\}$
and $f^\prime{[\perp]}M_+^\prime$ and $f^{\prime\prime}{\in}M_+^\prime$.\\
Thus,

\begin{equation*}
\begin{split}
		\sum|[f,f_{n_k}]|^2 & =\sum|[f,f_i]|^2+\sum|[f,f_j]|^2,~i\in\{n_k\}-\{l_k\},~j\in\{l_k\}\\
		& =\sum|[f^\prime,f_i]|^2+\sum|[f^{\prime\prime},f_j]|^2\\
		& \geq~A_1[f^\prime,f^\prime]+A_1^\prime[f^{\prime\prime},f^{\prime\prime}]\\
		& \geq~min\{A_1,A_1^\prime\}[f,f]
\end{split}
\end{equation*}

Hence $\{f_{n_k}:n_k\in{I_+^\prime}\}$ is a frame sequence in $(M_+,[~.~])$. Similarly we can show that $\{f_{n_k}:n_k\in{I_-^\prime}\}$ is a frame sequence in $(M_-,[~.~])$. Hence $\{f_{n_k}\}$ is a $J$-frame sequence in $\textbf{\textit{K}}$.\\
By similar arguments as above we can show that $\{f_{m_k}\}$ is also a $J$-frame sequence in $\textbf{\textit{K}}$.
\end{proof}
\begin{theorem}
Let $\{f_n:n\in{\mathbb{N}}\}$ be a $J$-frame for $\textbf{\textit{K}}$ with optimal bounds $B_2\leq{A}_2<0<A_1\leq{B}_1$
such that $f_n\neq~0~\forall~n\in{\mathbb{N}}$. If for every infinite increasing sequence $\{n_k\}\subset{\mathbb{N}}$, $\{f_{n_k}\}$
is a $J$-frame sequence with optimal bounds $B_2\leq{A}_2<0<A_1\leq{B}_1$, then $\{f_n\}$ is an exact $J$-frame.
\end{theorem}
\begin{proof}
Since $\{f_n:n\in{\mathbb{N}}\}$ is a $J$-frame for $\textbf{\textit{K}}$, so let $I_+=\{n\in{\mathbb{N}}:[f_n,f_n]\geq{0}\}$ and $I_-=\{n\in{\mathbb{N}}:[f_n,f_n]<0\}$. Now suppose $\{f_n\}$ is not exact. Without any loss of generality let us assume that there exists an $m\in{I_+}$ such that $f_m{\in~}\overline{span}\{f_i:i\in{I_+\setminus\{m\}}\}$. Let $\{n_k\}$ be an increasing sequence given by $n_k=k$, $k=~1,~2,~\ldots,m-1$ and $n_k=k+1$, $k =m,m+1,\ldots$. So for that $\{n_k\}\subset{\mathbb{N}}$, we have $\{f_{n_k}\}$ is also a $J$-frame sequence for $\textbf{\textit{K}}$ with bounds $B_2\leq{A}_2<0<A_1\leq{B}_1$. Let $M_+=\overline{span}\{f_n:n\in{I_+}\}$ and $\{f_n:n\in{I_+}\}$ is a frame for the Hilbert space $(M_+,[~.~])$ with optimal frame bounds $0<A_1\leq{B}_1$. Let $M_+^1=\overline{span}\{f_i:i\in{I_+\setminus\{m\}}\}$. Now $M_+^1$ is a closed subspace of the Hilbert space $(M_+,[~.~])$ such that $M_+^1=M_+$ and $\{f_i:i\in{I_+\setminus\{m\}}\}$ is a frame for $(M_+,[~.~])$ with the same frame bound.\\
Therefore, by frame inequality we have, $|[f,f_m]|^2=0~\forall~f\in~M_+$.
This gives $f_m=0$, which is a contradiction.\\
Hence the proof.
\end{proof}

{\bf Acknowledgement.}  The first author would like to thank  CSIR, Govt. of India for the finanical support. 

\bibliographystyle{amsplain}

\begin{thebibliography}{99}

\bibitem{ds} R. J. Duffin, A. C. Schaeffer, \textit{A Class of Nonharmonic Fourier Series}, 
Trans. Amer. Math. Soc. {\bf 72} (1952), no.2,
341--366.

\bibitem{dgm} I. Daubechies, A. Grossmann, Y. Meyer, \textit{Painless nonorthogonal expansions},
 J. Math. Phys., {\bf 27} (1986), 
1271--1283, doi: 10.1063/1.527388.

\bibitem{oc} O. Christensen, \textit{An Introduction to Frames and Riesz Bases}, Birkh$\ddot{\textmd{a}}$user Basel, 2002.

\bibitem{pgc} P. G. Casazza, \textit{The Art of Frame Theory}, Taiwanese J. Math. {\bf 4} (2000), no.2,
129--201.

\bibitem{pgcgt} P. C. Casazza, G. Kutyniok, \textit{Finite Frames: Theory and Applications}, 
 Birkh$\ddot{\textmd{a}}$user, (2013).

\bibitem{pcs} P. G. Casazza,  O. Christensen,  Diana T. Stoeva, \textit{Frame expansions in Separable Banach Spaces},
 J. Math. Anal. Appl., {\bf 307} (2005),
710--723.

\bibitem{phl} P. Casazza, D. Han,  D. Larson, \textit{Frames for Banach spaces}, 
Contemp. Math., {\bf 247} (1999),
149--182.

\bibitem{jb} J. Bognar, \textit{Indefinite inner product spaces}, Springer Berlin, 1974.

\bibitem{gmmm} J. I. Giribet, A. Maestripieri, F. Martínez Per\'{i}a,  P. G. Massey, \textit{On frames for Krein spaces}, J. Math. Anal. Appl., {\bf 393} Issue 1 (2012),
122--137.

\bibitem{koe} K. Esmeral, O. Ferrer, E. Wagner, \textit{Frames in Krein spaces arising from a non-regular W-metric},
Banach J. Math. Anal.
{\bf 9} No. 1 (2015),
1--16.

\bibitem{pahko} P. Acosta-Hum$\acute{a}$nez, K. Esmeral, O. Ferrer \textit{Frames of subspaces in Hilbert spaces with W-metric}, VERSITA. {\bf 23(2)} (2015), 5--22, 
doi:10.1515/auom-2015-0021.

\bibitem{ochr} O. Christensen,  \textit{Frames and the Projection Method}, Appl. and Comput. Har. Anal., {\bf 1} (1993), 50--53.

\bibitem{pssj} P.G. Casazza, S. Obeidat, S. Samarah and J.C. Tremain, \textit{Sums of Hilbert Space Frames}, 
J. Math. Analysis and Appl., 
{\bf 351}, Issue 2 (2009), 
579--585.

\bibitem{pgcml} P.G. Casazza and M. Leon, Projections of Frames, Proceedings of SPIE, 
{\bf 5914} (2005), pp. 591402,
 1--8.

\bibitem{cfklt} P.G. Casazza, M. Fickus, J. Kovacevic, M. Leon and J.C. Tremain, \textit{A Physical Interpretation of Tight Frames}, 
Harmonic Analysis and Applications (In Honor of John Benedetto), (2006),
51--76, doi:10.1007/0-8176-4504-7$\_$4

\bibitem{sgv} S. K. Kaushik, Ghanshyam Singh, Virender,  \textit{A Note on Frame Sequences in Hilbert Spaces}, Int. J. Contemp. Math. Sciences, {\bf 3} no.16, (2008), 
791--799.

\bibitem{jm} John Benedetto, Matthew Fickus,  \textit{Finite Normalized Tight Frames}, Advances in Computational Mathematics, 
{\bf 18} (2003),
357--385.

\bibitem{rbpcz} R. Balan, P.G. Casazza, C. Heil and Z. Landau, \textit{Excess of Parseval Frames}, Proceedings of SPIE, 
{\bf 5914} (2005), pp. 591406, 
1--7.

\bibitem{dkde} Deguang Han, Keri Kornelson, David Larson, Eric Weber, \textit{Frames for Undergraduates}, Student Mathematical Library, AMS {\bf 40} (2007).

\bibitem{isty} I. S. Iokhvidov, T. Ya. Azizov,  \textit{Linear operators in spaces with an indefinite metric}, John Wiley $\&$ sons. (1989).

\bibitem{ando} T. Ando, \textit{Projections in Krein spaces}, Linear Algebra and Appl., 
{\bf 431} (2009),
 2346--2358.



\end{thebibliography}

\end{document}